\documentclass[a4paper,12pt]{article}
\usepackage{latexsym,amsmath,amsthm,amssymb}
\usepackage{a4wide}
\usepackage{hyperref}
\usepackage{marginnote}
\usepackage{color}
\hypersetup{
pdftitle={General affine P\'olya-Szeg\"o principles},    
pdfauthor={Van Hoang Nguyen},
colorlinks = true,
linkcolor = magenta,
citecolor = blue,
}

\theoremstyle{plain}
\newtheorem{theorem}{Theorem}[section]
\newtheorem{proposition}[theorem]{Proposition}
\newtheorem{lemma}[theorem]{Lemma}
\newtheorem{corollary}[theorem]{Corollary}

\theoremstyle{definition}

\theoremstyle{remark}

\renewcommand{\thefootnote}{\arabic{footnote}}

\def\R{\mathbb R}

\def\Ent{\mathop{\rm Ent}\nolimits}

\def\al{\alpha}
\def\om{\omega}
\def\Om{\Omega}
\def\be{\beta}
\def\ga{\gamma}
\def\de{\delta}
\def\De{\Delta} 
\def\Gam{\Gamma}
\def\si{\sigma}
\def\lam{\lambda}
\def\vphi{\varphi}
\def\ep{\epsilon}
\def\na{\nabla}
\def\pa{\partial}
\def\la{\langle} 
\def\ra{\rangle} 
\def\lt{\left}
\def\rt{\right}

\def\mL{\mathcal{L}}
\def\mM{\mathcal{M}}
\def\mE{\mathcal{E}}

\def\mK{\mathcal K}
\def\mH{\mathcal H}

\def\mM{\mathcal M}

\def\dWp{\dot{W}^{p}(\R^n)}

\def\sp{\R^n}

\def\i0i{\int_0^\infty}

\def\t{\tilde}
\def\Klp{K_{\lam,p}(f)}
\def\Blp{B_{\lam,p}(f)}

\numberwithin{equation}{section}

\title{New approach to the affine P\'olya-Szeg\"o principle and the stability version of the affine Sobolev inequality}
\author{Van Hoang Nguyen\footnote{School of Mathematical Sciences, Tel Aviv University, Tel Aviv 69978, \textsc{Israel}.}}

\begin{document}
\maketitle


\renewcommand{\thefootnote}{}

\footnote{Email: vanhoang0610@yahoo.com}


\footnote{2010 \emph{Mathematics Subject Classification\text}: 26D07, 46E35, 52A40.}

\footnote{\emph{Key words and phrases\text}: affine P\'olya-Szeg\"o principle, Busemann-Petty centroid inequality, affine Sobolev-type inequalities, stability estimates}

\renewcommand{\thefootnote}{\arabic{footnote}}
\setcounter{footnote}{0}

\begin{abstract}
Inspired by a recent work of Haddad, Jim\'enez and Montenegro, we give a new and simple approach to the recently established general affine P\'olya-Szeg\"o principle. Our approach is based on the general $L_p$ Busemann-Petty centroid inequality and does not rely on the general $L_p$ Petty projection inequality or the solution of the $L_p$ Minkowski problem. A Brothers-Ziemer-type result for the general affine P\'olya-Szeg\"o principle is also established. As applications, we reprove some sharp affine Sobolev-type inequalities and settle their equality conditions. We also prove a stability estimate for the affine Sobolev inequality on functions of bounded variation by using our new approach. As a corollary of this stability result, we deduce a stability estimate for the affine logarithmic--Sobolev inequality.
\end{abstract}

\section{Introduction}
The first (symmetric) affine P\'olya-Szeg\"o principle was established in \cite{CLYZ09} by Cianchi, Lutwak, Yang and Zhang. More recently, Haberl, Schuster and Xiao \cite{HSX12} obtained an asymmetric version of the affine P\'olya-Szeg\"o principle and proved that it is stronger and directly implies the symmetric form of Cianchi et al. The same argument also shows that the asymmetric affine P\'olya-Szeg\"o principle directly implies the general affine P\'olya-Szeg\"o principle which was first explicitly stated by Wang \cite{TW15} (which includes the symmetric and asymmetric versions as special cases). These affine P\'olya-Szeg\"o-type principles state that the general (also, symmetric and asymmetric) affine energy of a function on $\R^n$ does not increase under symmetric rearrangement. They are affine analogues of the classical P\'olya-Szeg\"o principle \cite{PS51} which plays a fundamental role in the solution of several variational problems from different areas such as isoperimetric inequalities, sharp forms of Sobolev inequalities and sharp a priori estimates of solutions to second-order elliptic or parabolic boundary value problems; see, for example, \cite{BZ88, Kaw85,Kaw86,Ke06,Ta76,Ta93}. It is known that the affine P\'olya-Szeg\"o-type principles are stronger and directly imply the classical P\'olya-Szeg\"o principle. Among their applications are a number of sharp affine Sobolev-type inequalities \cite{CLYZ09,LYZ02,Zhai11,TW15} (like Sobolev, Moser-Trudinger, Morrey-Sobolev, Gagliardo-Nirenberg, and logarithmic-Sobolev inequalities) which are also stronger than their classical Euclidean counterparts. The proof of the general (symmetric and asymmetric) affine P\'olya-Szeg\"o principle relies on tools from the $L_p$ Brunn-Minkowski theory of convex bodies (see \cite{Lut93,Lut96,LYZ00,LYZ02, LYZ04}). In particular, on the $L_p$ Petty projection inequality \cite{LYZ00} (see \cite{CG02} for an alternative proof) and on the solution of the normalized $L_p$ Minkowski problem \cite{LYZ04}.


The main aim of the present paper is to give a new proof for these affine P\'olya-Szeg\"o-type principles. Our approach is based on a recent work of Haddad, Jim\'enez and Montenegro (see \cite{HJM15}) where they give a new proof of some sharp (symmetric) affine Sobolev-type inequalities (like Sobolev, Gagliardo-Nirenberg and logarithmic-Sobolev inequalities \cite{LYZ02,Zhai11}) by using the $L_p$ Busemann-Petty centroid inequality \cite{LYZ00}. We show that their method also can be applied to general cases considered in \cite{HS09b,HSX12,TW15}, and hence gives an alternative proof for the results in \cite{CLYZ09,HSX12,TW15}. We emphasize here that neither the (general) $L_p$ Petty projection inequality \cite{HS09,LYZ00} nor the solution of the normalized $L_p$ Minkowski problem for symmetric convex bodies or for polytopes is used in our proof while they play crucial and important roles in the proofs given in \cite{CLYZ09,HSX12,TW15}.

One advantage of our approach is that it enables us to prove a Brothers-Ziemer-type result for the general affine P\'olya-Szeg\"o principle which was left open in \cite{HSX12,TW15} by restriction of their method. Such a result for the (symmetric) affine P\'olya-Szeg\"o principle was recently proved in \cite{TW13} by Wang. In fact, Wang \cite{TW13} proved a bit more, namely a stability estimate for the affine P\'olya-Szeg\"o principle by using a recent stability result for the $L_p$ Petty projection inequality establisheb by B\"or\"oczky \cite{KB13} and a quantitative P\'olya-Szeg\"o pinciple for convex symmetrization \cite{ER09}. We do not know how to obtain the same result for the general case of the affine P\'olya-Szeg\"o principle considered in this paper. However, as a consequence of our Brothers-Ziemer-type result, we identify all extremal functions for the new affine Sobolev-type inequalities established in \cite{HS09,HS09b,TW15}.

For $p\geq 1$ and $n\geq 2$, let $W^{1,p}(\R^n)$ denote the space of real-valued $L^p$ functions on $\R^n$ with weak $L^p$ partial derivatives. We use $|\cdot|$ and $B_2^n$ to denote the standard Euclidean norm on $\R^n$ and its unit ball, respectively. We write $\|\cdot\|_p$ for the usual $L^p$ norm of a function $f$ on $\R^n$. For $f\in W^{1,p}(\R^n)$, we set
$$\|\na f\|_p = \lt(\int_{\R^n} |\na f(x)|^p dx\rt)^{1/p}.$$

Given a function $f\in W^{1,p}(\R^n)$, its distributional function $\mu_f:[0,\infty)\to [0,\infty)$ is defined by
$$\mu_f(t) = V(\{x\in\R^n\, :\, |f(x)| > t\}),$$
where $V$ denotes the Lebesgue measure on $\R^n$. The decreasing rearrangement $f^*: [0,\infty) \to [0,\infty]$ of $f $ is defined by
$$f^*(s) = \inf\{t > 0\, :\, \mu_f(t) \leq s\}.$$
Let $K$ denote a compact convex subset of $\R^n$ containing the origin in its interior. We write $\t{K}$ to denote the dilation of $K$ of volume $\om_n:=V(B_2^n)$, its Minkowski functional is defined by
$$\|x\|_{\t{K}} = \inf\{\lam > 0\, :\, x\in \lam \t{K}\}.$$
The convex symmetrization $f^K$ of $f$ with respect to $K$ is defined as follows (see \cite{AFTL97})
$$f^K(x) = f^*(\om_n \|x\|_{\t{K}}^n).$$
When $K$ is an origin-centered Euclidean ball, we obtain the symmetric rearrangement function $f^\star$ of $f$, that is,
$$f^\star(x) = f^*(\om_n |x|^n).$$

The classical P\'olya-Szeg\"o principle states that if $f\in W^{1,p}(\R^n)$ for some $p \geq 1$, then $f^\star \in W^{1,p}(\R^n)$ and 
\begin{equation}\label{eq:classicalPS}
\|\na f\|_p \geq \|\na f^\star\|_p.
\end{equation}
In the general affine P\'olya-Szeg\"o principle, the $L^p$ norm of the Euclidean length of the weak gradient is replaced by a general $L_p$ affine energy defined for $f\in W^{1,p}(\R^n)$ by
\begin{equation}\label{eq:affineenergy}
\mE_{\lam,p}(f) = c_{n,p} \lt(\int_{S^{n-1}} \lt(\int_{\R^n} \lt((1-\lam) (D_uf(x))_+^p + \lam (D_u f(x))_-^p\rt)dx\rt)^{-n/p} du\rt)^{-1/n},
\end{equation}
where $\lam \in [0,1]$, $c_{n,p} = (n\om_n)^{1/n}(n\om_n\om_{p-1}/\om_{n+p-2})^{1/p}$, $D_uf$ is the directional derivative of $f$ in the direction $u$, and $(D_uf)_+ = \max\{D_uf,0\}$ and $(D_uf)_- =\max\{-D_u f,0\}$. Note that the constant $c_{n,p}$ is chosen such that 
\begin{equation}\label{eq:normalized}
\mE_{\lam,p}(f^\star) = \|\na f^\star\|_p.
\end{equation}
We emphasize the remarkable and important fact that $\mE_{\lam,p}$ is invariant under volume preserving affine transformations on $\R^n$, in contrast to $\|\na f\|_p$ which is invariant under rigid motions only. Note that $\mE_{1/2,p}(f)$ is exactly the $L_p$ affine energy $\mE_p(f)$ defined by Cianchi, Lutwak, Yang and Zhang \cite{CLYZ09}, and $\mE_{0,p}(f)$ is the asymmetric $L_p$ affine energy $\mE_p^+(f)$ defined by Haberl, Schuster and Xiao \cite{HSX12}.

The first main result of this paper reads as follows:

\begin{theorem}\label{maintheorem}
If $p > 1$ and $f\in W^{1,p}(\R^n)$, then 
\begin{equation}\label{eq:PSprin}
\mE_{\lam,p}(f) \geq \mE_{\lam,p}(f^\star).
\end{equation}
Moreover, if $f$ is a nonnegative function such that
\begin{equation}\label{eq:cond}
V(\{x\,:\, |\na f^\star(x)| =0\} \cap \{x\, :\, 0< f^\star(x) < {\rm ess }\sup f\}) =0,
\end{equation}
where ${\rm ess } \sup f$ denotes the essential supremum of $f$, then equality holds in \eqref{eq:PSprin} if and only if there exists $x_0\in \R^n$ such that $f(x) = f^E(x+x_0)$ a.e. on $\R^n$, here $E$ is an origin-centered ellipsoid (i.e., an image of $B_2^n$ under an invertible linear map).
\end{theorem}  
It is easy to see that $\mE_{\lam,p}(f) = \mE_{1-\lam,p}(f)$ and that the function $\lam \to \mE_{\lam,p}(f)$ is concave on $[0,1]$. Consequently, we have 
\begin{equation*}\label{eq:concave}
\mE_{0,p}(f) \leq \mE_{\lam,p}(f) \leq \mE_{1/2,p}(f),\quad\forall\, \lam \in [0,1].
\end{equation*}
Thus, the asymmetric affine P\'olya-Szeg\"o principle is the strongest in the family of general affine P\'olya-Szeg\"o principles, while the affine P\'olya-Szeg\"o principle is the weakest one. It follows from Minkowski's inequality that
\begin{equation}\label{eq:compa}
\mE_{\lam,p}(f) \leq \|\na f\|_p.
\end{equation}
In view of \eqref{eq:normalized}, \eqref{eq:PSprin} and \eqref{eq:compa}, we obtain
\begin{equation}\label{eq:series}
\|\na f^\star\|_p = \mE_{\lam,p}(f^\star) \leq \mE_{\lam,p}(f) \leq \|\na f\|_p,\quad \forall\, \lam\in[0,1].
\end{equation} 

The second statement in Theorem \ref{maintheorem} is  a Brothers-Ziemer-type result for the general affine P\'olya-Szeg\"o principle. This result yields a necessary condition to ensure that the general $L_p$ affine energy of a function on $\R^n$ is equal to that of its symmetric rearrangement. The condition \eqref{eq:cond} is equivalent to the absolute continuity of $\mu_f$ and it can not be removed. In fact, Brothers and Ziemer constructed in \cite{BZ88} a smooth function $f$ such that $\|\na f\|_p = \|\na f^\star\|_p$ without $f$ being a translation of $f^\star$, if one does not assume condition \eqref{eq:cond}. For such a function $f$, we must have $\mE_{\lam,p}(f) = \mE_{\lam,p}(f^\star)$ by \eqref{eq:series}. However, $f$ is not a translation of $f^E$ for any origin-centered ellipsoid $E$.


A function $f$ on $\R^n$ is said to be of bounded variation if $f \in L^{n/(n-1)}(\R^n)$ and its weak gradient $Df$ is an $\R^n-$valued Radon measure of bounded variation. Let $|Df|$ denote the total variation of $Df$, and let $\si_f$ denote the Radon-Nikodym derivative of $Df$ with respect to $|Df|$. Then we have $|\si_f| = 1$ $|Df|-$almost everywhere (a.e.), and $d(Df)(x) = \si_f(x) d(|Df|)(x)$ (see \cite{EG}). Let $BV(\R^n)$ denote the space of functions of bounded variation on $\R^n$. For each $f\in BV(\R^n)$, its $L_1$ affine energy is defined as follows (see \cite{TW12})
$$\mE_{1}(f) = \frac{c_{n,1}}2\, \lt(\int_{S^{n-1}} \lt(\int_{\R^n} |\la u,\si_f(x)\ra| d(|Df|)(x)\rt)^{-n}\rt)^{-1/n}.$$  

The sharp affine Sobolev inequality for functions on $BV(\R^n)$ states that for any $f\in BV(\R^n)$ we have
\begin{equation}\label{eq:ZhangWang}
\mE_1(f) \geq n \om_n^{1/n} \|f\|_{n'},
\end{equation}
where $n' = n/(n-1)$. This inequality was first proved by Zhang \cite{GZ99} for compactly supported $C^1$ functions on $\R^n$, and then was extended to functions of bounded variation by Wang \cite{TW12}. The constant $n\om_n^{1/n}$ in \eqref{eq:ZhangWang} is sharp and it is attained by the characteristic functions of ellipsoids. By Minkowski's inequality, it is easy to see that 
$$\|f\|_{BV} = |Df|(\R^n) \geq \mE_1(f),$$ thus the affine Sobolev inequality \eqref{eq:ZhangWang} is stronger than the classical $L^1$ Sobolev inequality. It is a well known fact that the classical $L_1$ Sobolev inequality is the functional form of the classical isoperimetric inequality. It was first shown by Zhang \cite{GZ99} that the affine Sobolev inequality \eqref{eq:ZhangWang} is the functional form of an affine isoperimetric inequality (more precisely, the Petty projection inequality) which is also stronger than the classical isoperimetric inequality.

Our next result establishes a stability estimate for inequality \eqref{eq:ZhangWang}. More precisely, it was determined how to measure the distance (in some sense) from $f$ to the set of extremal functions of \eqref{eq:ZhangWang} in terms of the deviation between $\mE_1(f)$ and $n\om_n^{1/n} \|f\|_{n'}$. To state this result, let us introduce some notation. For $f\in BV(\R^n)$, we define
$$\de_a(f) = \frac{\mE_1(f)}{n\om_n^{1/n} \|f\|_{n'}} -1,$$
if $f\not= 0$ and $\de_a(0) = 0$. We remark that $\de_a(\cdot)$ is invariant under the action of invertible affine transformations. We call it the affine Sobolev deficit functional. The class of extremal functions for \eqref{eq:ZhangWang} is denoted by $\mM$, that is,
$$\mM=\{g_{a,x_0,\psi,r} = a \chi_{x_0 + ar\psi(B_2^n)}\,:\, a \not = 0, r>0, x_0\in \R^n, \psi\in SL_n\},$$
where $SL_n$ denotes the set of $n\times n$ volume preserving linear transformations.

For $f\in BV(\R^n)$, we define the (normalized) distance of $f$ from $\mM$ by
$$d_a(f,\mM) = \inf\lt\{\frac{\|f-g_{a,x_0,\psi,r}\|_{n'}^{n'}}{\|f\|_{n'}^{n'}}\, :\, g_{a,x_0,\psi,r} \in \mM,\, \|f\|_{n'} = \|g_{a,x_0,\psi,r}\|_{n'}\rt\},$$
if $f\not= 0$ and $\lam (0) = 0$.

The second main result of this paper is a stability estimate for \eqref{eq:ZhangWang}.
\begin{theorem}\label{maintheo1}
For $n\geq 2$, there exists a positive constant $\al(n)$ depending only on $n$ such that
\begin{equation}\label{eq:Stab}
d_a(f,\mM) \leq \al(n) \de_a(f)^{1/cn},
\end{equation}
for any $f\in BV(\R^n)$, with $c = 1680$.
\end{theorem}
Theorem \ref{maintheo1} is proved by exploiting the recently established quantitative anisotropic Sobolev inequality of Figalli, Maggi and Pratelli \cite{FMP13} and a stability estimate for the Busemann--Petty centroid inequality which will be presented in Section \S5 below. The latter inequality is derived from a stability estimate for the $L_p$ Petty projection inequality due to B\"or\"oczky \cite{KB13}, the class reduction technique introduce by Lutwak \cite{Lut90}, and an improved dual mixed volume inequality proved in Section \S2. We should emphasize here that the order of $\de_a(f)$ in \eqref{eq:Stab} is not optimal. We believe that its sharp value should be $1/2$ as in the result of Figalli, Maggi and Pratelli \cite{FMP13} for the anisotropic Sobolev inequality which strengthened the previous results in \cite{Ci06,FMP07}. The study of stability estimates for inequalities in analysis and geometry recently has attracted lots of attention by many mathematicians and has become an important field in mathematical research. We refer the reader to \cite{BB10,BB11,KB10,KB13,Ci06,CEFT08,CFMP09,ER09,FiMP09,FMP10, FMP13,FJ14,FMP07,FMP08} and references therein for more background and results in this direction.

The rest of this paper is organized as follows. In the next section we recall some background material from the $L_p$ Brunn-Minkowski theory of convex bodies. In Section \S3, we prove Theorem \ref{maintheorem}. Section \S4 is devoted to reproving several sharp affine Sobolev-type inequalities and obtain their equality conditions as corollaries of Theorem \ref{maintheorem}. In the last Section \S5, we give a proof of the stability estimate for the affine Sobolev inequality on $BV(\R^n)$ (that is, Theorem \ref{maintheo1}).


\section{Background Material}
For quick later reference we recall in this section some background material from the $L_p$ Brunn-Minkowski theory of convex bodies. This theory has its origins in the work of Firey \cite{Firey62} and was further developed by Lutwak and many authors (see \cite{CG02,HS09,HS09b,HLYZ05,Lut93,Lut96,LYZ00,LYZ04,LYZ06,TW15}). We also list some basic facts from real analysis which we need in our proofs of Theorem \ref{maintheorem} and Theorem \ref{maintheo1}.

A convex body is a compact convex subset of $\R^n$ with nonempty interior. We denote by $\mK^n$ the set of convex bodies in $\R^n$ endowed with the Hausdorff metric, by $\mK_0^n$ the set of convex bodies containing the origin in their interiors, and by $GL_n$ the set of invertible linear transformations of $\R^n$. It is well known that each convex body $K\in \mK^n$ is uniquely determined by its support function defined by
$$h(K,x) = \sup\{\la x,y\ra\, :\, y\in K\},\quad x\in \R^n.$$
Note that $h(K,\cdot)$ is positively homogeneous of degree one and subadditive. Conversely, every function with these properties is the support function of a unique compact convex set.

If $K\in \mK_0^n$, then the polar body $K^*$ of $K$ is defined by
$$K^* =\{x\in \R^n\, :\, \la x,y\ra \leq 1\, \text{ for all }\, y\in K\}.$$
From the polar formula for volume, it follows that the $n-$dimensional Lebesgue measure $V(K^*)$ of the polar body $K^*$ is given by
\begin{equation}\label{eq:polarvolume}
V(K^*) = \frac1n \int_{S^{n-1}} h(K,u)^{-n} du,
\end{equation}
where the integration is with respect to spherical Lebesgue measure. We refer the readers to the book \cite{RS} for more background from the theory of convex bodies.

For real $p\geq 1$ and $\al, \be > 0$, the $L_p$ Minkowski-Firey combination of $K,L \in \mK_0^n$ is the convex body $\al \cdot K +_p \be\cdot L$ whose support function is defined by (see \cite{Firey62})
$$h(\al \cdot K +_p \be\cdot L,\cdot)^p = \al h(K,\cdot)^p + \be h(L,\cdot)^p.$$
The $L_p$ mixed volume $V_p(K,L)$ of $K,L\in \mK_0^n$ was defined in \cite{Lut93} by
$$V_p(K,L) = \frac pn \lim_{\ep\to 0^+} \frac{V(K +_p \ep\cdot L) - V(K)}{\ep}.$$
It is clear that $V_p(K,K) = V(K)$ for every $K\in \mK_0^n$. It was also shown in \cite{Lut93} that for all convex bodies $K,L\in \mK_0^n$,
\begin{equation}\label{eq:Lpmixed}
V_p(K,L) = \frac1n \int_{S^{n-1}} h(L,u)^p h(K,u)^{1-p} dS(K,u),
\end{equation}
where the measure $S(K,\cdot)$ on $S^{n-1}$ is the classical surface area measure of $K$. Recall that for a Borel set $\om \subset S^{n-1}$, $S(K,\om)$ is the $(n-1)-$dimensional Hausdorff measure of the set of all boundary points of $K$ for which there exists a normal vector of $K$ belonging to $\om$.

A compact subset $K$ of $\R^n$ is called star-shaped (about the origin) if for any $x\in K$, the interval $\{tx\,:\, t\in [0,1]\}$ is contained in $K$. In this case, the radial function of $K$, $\rho(K,\cdot): \R^n\setminus\{0\}\to [0,\infty)$, is defined by
$$\rho(K,x) = \max\{\lam \geq 0\, :\, \lam x \in K\},\quad x\not=0.$$
If $\rho(K,\cdot)$ is positive and continuous, we call $K$ a star body (about the origin). Two star bodies $K$ and $L$ are said to be dilates if $\rho(K,u)/\rho(L,u)$ is independent of $u\in S^{n-1}$.

If $K\in \mK_0^n$ then it is easy to prove that
$$h(K^*,\cdot) = 1/\rho(K,\cdot),\quad\text{and}\quad \rho(K^*,\cdot) = 1/h(K,\cdot).$$
For star bodies $K,L$ and $\ep > 0$, the $L_p$ harmonic radial combination $K \t{+}_p \ep\cdot L$ is the star body defined by
$$\rho(K \t{+}_p\, \ep\cdot L,\cdot)^{-p} = \rho(K,\cdot)^{-p} +\ep  \rho(L,\cdot)^{-p}.$$
The dual mixed volume $\t{V}_{-p}(K,L)$ of the star bodies $K,L$ can be defined by
$$\t{V}_{-p}(K,L) = -\frac{p}n \lim_{\ep\to 0} \frac{V(K \t{+}_p\, \ep\cdot L) -V(K)}{\ep}.$$
We have the following integral representation for the dual mixed volume $\t{V}_{-p}(K,L)$ of the star bodies $K,L$ (see \cite{LYZ00}),
\begin{equation}\label{eq:dualmixed}
\t{V}_{-p}(K,L) = \frac1n \int_{S^{n-1}} \rho(K,u)^{n+p} \rho(L,u)^{-p} du.
\end{equation}
Note that $\t{V}_{-p}(K,K) = V(K)$ for each star body $K$. By the H\"older inequality, we have
\begin{equation}\label{eq:dualineq}
\t{V}_{-p}(K,L) \geq V(K)^{(n+p)/n} V(L)^{-p/n},
\end{equation}
with equality if and only if $K, L$ are dilates.

An improved version of \eqref{eq:dualineq} which may be of independent interest is given in the next proposition and is used in the proof of Theorem \ref{maintheo1}. 
\begin{proposition}\label{improved}
If $K$ and $L$ are star bodies in $\R^n$, then
\begin{equation}\label{eq:improved}
\frac{\t{V}_{-p}(K,L)}{V(K)^{(n+p)/n} V(L)^{-p/n}} - 1 \geq \frac{p}{8n} \lt(\frac{V(K\De(\ga L))}{V(K)}\rt)^2, \quad \ga = \lt(\frac{V(K)}{V(L)}\rt)^{1/n}.
\end{equation}
\end{proposition}
\begin{proof}
We can assume, without loss of generality, that $V(K) = V(L) = 1$ by the homogeneity of the functionals in \eqref{eq:improved}, hence
$$\int_{S^{n-1}} \rho(L,u)^n \frac{du}n =1.$$
According to Theorem $1.3$ in \cite{VHN15a}, we have
\begin{align*}
\t{V}_{-p}(K,L) & = \int_{S^{n-1}} \lt(\frac{\rho(K,u)^n}{\rho(L,u)^n}\rt)^{(n+p)/n} \rho(L,u)^n \frac{du}n\\
&\geq 1 + \frac{p}{2n} \lt(\int_{S^{n-1}} |\rho(K,u)^n -\rho(L,u)^n| \frac{du}n\rt)^2\\
&= 1 + \frac p{8n} \lt(\int_{S^{n-1}} (\max\{\rho(K,u)^n ,\rho(L,u)^n\} - \min\{\rho(K,u)^n ,\rho(L,u)^n\})\frac{du}n\rt)^2\\
&= 1 + \frac p{8n} \lt(\int_{S^{n-1}} (\rho(K\cup L, u)^n - \rho(K\cap L,u)^n) \frac{du}n\rt)^2\\
&= 1 + \frac{p}{8n} \lt(V(K\cup L) - V(K\cap L)\rt)^2\\
&= 1 + \frac p{8n} V(K \De L)^2,
\end{align*}
which is exactly \eqref{eq:improved}.
\end{proof}

For $K\in \mK_0^n$, the asymmetric $L_p$ centroid body of $K$ is a natural notion from the theory of $L_p$ Minkowski valuations as was first discovered by Ludwig \cite{Ludwig} (see also \cite{Para1,Para2}). It is defined by
\begin{equation}\label{eq:asymcentroid}
h(\Gam_p^+K,u)^p = \frac1{\al_{n,p} V(K)} \int_K \la u,y\ra_+^p dy,
\end{equation}
where $\al_{n,p} = \om_{n+p-2}/((n+p)\om_n \om_{p-1})$ is a normalizing constant such that $\Gam_{p}^+B_2^n = B_2^n$. For each $\lam\in [0,1]$, the general $L_p$ centroid body of $K\in \mK_0^n$ is the convex body
\begin{equation}\label{eq:general}
\Gam_{\lam,p}K = (1-\lam)\cdot \Gam_p^+ K +_p \lam\cdot \Gam_p^-K,
\end{equation}
where $\Gam_p^-K = \Gam_p^+(-K)$, thus we have
\begin{equation}\label{eq:generalsupport}
h(\Gam_{\lam,p}K,u)^p = \frac1{\al_{n,p} V(K)} \int_K ((1-\lam)\la u,y\ra_+^p + \lam \la u,y\ra_-^p) dy.
\end{equation}
Note that in the symmetric case $\lam = 1/2$, we recover the $L_p$ centroid body $\Gam_pK$ introduced by Lutwak et al. \cite{LYZ00}.

The following general affine $L_p$ Busemann-Petty centroid inequality established in \cite{HS09} plays the crucial role in our proof of the general affine P\'olya-Szeg\"o principle.
\begin{theorem}\label{HS09}
If $p\geq 1$ and $K\in \mK_0^n$, then 
\begin{equation}\label{eq:HS09}
V(\Gam_{\lam,p}K) \geq V(K),
\end{equation}
with equality if and only if $K$ is an origin-centered ellipsoid.
\end{theorem}  
Although this inequality was formulated in \cite{HS09} for dimensions $n\geq 3$ and $p> 1$, we remark that it also holds true in dimension $n=2$ with the same proof as the one in  \cite{HS09}. It was also shown in \cite{HS09} that inequality \eqref{eq:HS09}, for $p > 1$, strengthens and directly implies the affine $L_p$ Busemann-Petty centroid inequality established by Lutwak et al. in \cite{LYZ00}, namely if $K\in \mK_0^n$, then 
$$V(\Gam_{\lam,p}K) \leq V(\Gam_pK).$$

We turn now to the second tool needed in our proof of the general affine P\'olya-Szeg\"o principle. Convex symmetrization was introduced by Alvino, Ferone, Trombetti and Lions in \cite{AFTL97} and further developed in \cite{ET04,ER09,FV04}. Similar to the case of symmetric rearrangement, a P\'olya-Szeg\"o principle also holds true for convex symmetrization. Moreover, we have the following results from \cite{FV04}.

\begin{theorem}\label{Convexsymme}
Let $K$ be a convex body in $\mK_0^n$ and let $p \in (1,\infty)$. For any function $f \in W^{1,p}(\R^n)$, we have 
\begin{equation}\label{eq:Convexsymme}
\int_{\R^n} h(K, -\na f(x))^p dx \geq \int_{\R^n} h(K, -\na f^K(x))^p dx.
\end{equation}
Moreover, if $f$ is a nonnegative function in $W^{1,p}(\R^n)$ such that 
\begin{equation}\label{eq:Condconvex}
V(\{x\,:\, |\na f^K(x)| =0\}\cap\{x\, :\, 0 < f^K(x) < {\rm ess}\sup f\}) =0,
\end{equation}
then equality holds in \eqref{eq:Convexsymme} if and only if there exists $x_0\in \R^n$ such that $f(x) = f^K(x+x_0)$ a.e. on $\R^n$. 
\end{theorem}


\section{Proof of Theorem \ref{maintheorem}}
In this section, we give the proofs of Theorem \ref{maintheorem} and Theorem \ref{maintheo1}. Our proofs follow the ideas from the recent paper of Haddad, Jim\'enez and Montenegro \cite{HJM15} where these authors gave a new proof of several sharp affine Sobolev-type inequalities (such as Sobolev, Gagliardo-Nirenberg and logarithmic-Sobolev inequalities) based on the affine $L_p$ Busemann-Petty centroid inequality of Lutwak et al. \cite{LYZ00}, and the results of Cordero-Erausquin, Nazaret and Villani \cite{CNV} and of Gentil \cite{Gen}.

Let $f\in W^{1,p}(\R^n)$ be not $0$ a.e. on $\R^n$ and let $p >1$ be a real number. For each $x\in \R^n$ let us define
$$\|x\|_{p,\lam,f} = \lt(\int_{\sp} \lt((1-\lam) \la x,\na f(y)\ra_+^p + \lam \la x, \na f(y)\ra_-^p\rt) dy\rt)^{1/p},$$
and
$$\Blp =\{x\in \R^n\, :\, \|x\|_{p,\lam,f} \leq 1\}.$$
The following lemma is elementary.
\begin{lemma}\label{key1}
Let $f\in W^{1,p}(\R^n)$ be not $0$ a.e. on $\R^n$, then $\Blp\in \mK_0^n$.
\end{lemma}
\begin{proof}
Since the function $\|\cdot\|_{p,\lam,f}$ is positively homogeneous of degree one and convex, the set $\Blp$ is a closed, convex subset of $\R^n$. Thus, to prove $\Blp\in \mK_0^n$, it is enough to show that there exist constants $C,c > 0$ such that
\begin{equation}\label{eq:Blp}
c \leq \|u\|_{p,\lam,f} \leq C,\quad \forall\, u\in S^{n-1}.
\end{equation}
It is easy to check that $\|x\|_{p,\lam,f} \leq \|\na f\|_p |x|$ for every $x\in \R^n$. Thus, we can choose $C = \|\na f\|_p$ in \eqref{eq:Blp}.
 
For the existence of $c >0$, it is enough to show that $\|u\|_{p,\lam,f} > 0$ for all $u\in S^{n-1}$ by the continuity of the function $\|\cdot\|_{p,\lam,f}$ on $S^{n-1}$. On the other hand, since $(D_uf)_- = (D_{-u} f)_+$, it suffices to prove that $\int_{\R^n} ((D_uf)_+(y))^p dy >0$ for any $u\in S^{n-1}$. We argue by contradiction. If there exists $u \in S^{n-1}$ such that
$$\int_{\R^n} ((D_uf)_+(y))^p dy =0,$$ 
then $(D_uf)_+(y) =0$ a.e. on $\R^n$, or equivalently $D_uf(y) \leq 0$ a.e. on $\R^n$. Let $\psi \in C_0^\infty(\R^n)$ such that $\psi \geq 0$ and $\int_{\R^n} \psi(y) dy = 1$. For each $\de > 0$, let us define $\psi_\de(x) = \de^{-n} \psi(x/\de)$ and $f_\de = f\star \psi_\de$. We then have $f_\de \in C^\infty(\R^n)$, $f_\de \to f$ in $L^p(\R^n)$ as $\de \to 0^+$ and $D_u f_\de(x)\leq 0$ for all $x\in \R^n$. For each fixed $y \in u^\bot$, we have $\frac{\pa}{\pa t} f_\de(y + tu) = D_uf_\de(y + tu) \leq 0$. Thus $f_\de(y+tu)$ is a decreasing function of $t$ . Since $f\in L^p(\R^n)$, for a.e $y\in u^\bot$, we have
$$\int_\R |f_\de(y+tu)|^p dt < \infty.$$
For such a $y\in u^\bot$, we must have $f_\de(y+t u) = 0$ for all $t\in \R$ by monotonicity. Consequently, $f_\de(x) = 0$ for all $x\in \R^n$. By letting $\de \to 0^+$, we obtain $f(x) = 0$ for a.e. on $\R^n$, which contradicts the assumption on $f$ and hence finishes the proof of this lemma.
\end{proof}

Since $\|\cdot\|_{p,\lam, f}$ is the Minkowski functional associated with $\Blp$, we have
$$\| \cdot\|_{p,\lam,f} = h((\Blp)^*, \cdot).$$ 
Combining \eqref{eq:polarvolume} and the definition of the general $L_p$ affine energy \eqref{eq:affineenergy} shows that
\begin{equation}\label{eq:affineenergy1}
\mE_{\lam,p}(f) = c_{n,p} (n V(\Blp))^{-1/n}.
\end{equation} 

We next define for $x\in \R^n$,
$$|||x|||_{p,\lam,f} = \lt(\int_{S^{n-1}} \|\xi\|_{p,\lam,f}^{-n-p} \lt((1-\lam) \la x,\xi\ra_+^p +\lam \la x,\xi\ra_-^p\rt)d\xi\rt)^{1/p}.$$
This is well defined by \eqref{eq:Blp}. Since the function $|||\cdot|||_{p,\lam,f}$ is positively homogeneous of degree one and convex, there exists a unique convex body $\Klp$ whose support function is $|||\cdot|||_{p,\lam,f}$. The next lemma gives us a useful relation between $\Klp$ and the general $L_p$ centroid body of $\Blp$. More precisely, we have the following.
\begin{lemma}\label{key2}
Let $f\in W^{1,p}(\R^n)$ be not $0$ a.e. on $\R^n$. Then 
\begin{equation}\label{eq:relationkey}
\Klp = \lt((n+p)\al_{n,p} V(B_{\lam,p}(f))\rt)^{1/p}  \Gam_{\lam,p}B_{\lam,p}(f).
\end{equation}
In particular, $\Klp \in \mK_0^n$.
\end{lemma}
\begin{proof}
Indeed, using polar coordinates, we have for any $u \in S^{n-1}$,
\begin{align*}
h(K_{\lam,p}(f),u)^p& = \int_{S^{n-1}}\|\xi\|_{p,\lam,f}^{-n-p} ((1-\lam)\la \xi,u\ra_+^p + \lam \la \xi,u\ra_-^p)\, d\xi\\
&=(n+p) \int_{S^{n-1}}\int_0^{\|\xi\|_{p,\lam,f}^{-1}} ((1-\lam)\la \xi,u\ra_+^p + \lam \la \xi,u\ra_-^p) r^{n+p-1} dr d\xi\\
&= (n+p) \int_{B_{\lam,p}(f)}((1-\lam)\la u, y\ra_+^p + \lam \la u, y\ra_-^p) dy\\
&= (n+p) \alpha_{n,p} |B_{\lam,p}(f)| \, h(\Gam_{\lam,p}B_{\lam,p}(f),u)^p,
\end{align*}
which implies \eqref{eq:relationkey}.
\end{proof}

We also need the following lemma.
\begin{lemma}\label{energy}
Let $f\in W^{1,p}(\R^n)$ be not $0$ a.e. on $\R^n$. Then
\begin{equation}\label{eq:intkey}
\int_{\R^n} h(K_{\lam,p}(f),\na f(y))^p dy = \lt(\frac{\mE_{\lam,p}(f)}{c_{n,p}}\rt)^{-n}.
\end{equation}
\end{lemma}
\begin{proof}
By Fubini's theorem and \eqref{eq:affineenergy1}, we have
\begin{align*}
\int_{\R^n} h(K_{\lam,p}(f),\na f(y))^p dy&= \int_{\R^n} \int_{S^{n-1}} \|\xi\|_{p,\lam,f}^{-n-p} ((1-\lam)\la \xi, \na f(y)\ra_+^p + \lam\la \xi, \na f(y)\ra_-^p) d\xi dy\\
&=\int_{S^{n-1}}\|\xi\|_{p,\lam,f}^{-n-p} \int_{\R^n}((1-\lam)\la \xi, \na f(y)\ra_+^p + \lam\la \xi, \na f(y)\ra_-^p) dy d\xi\\
&= \int_{S^{n-1}} \|\xi\|_{p,\lam,f}^{-n} d\xi\\
&= n|B_{\lam,p}(f)|\\
& = \lt(\frac{\mE_{\lam,p}(f)}{c_{n,p}}\rt)^{-n}.
\end{align*}
\end{proof}
Using Lemma \ref{energy}, we obtain the following estimate.
\begin{proposition}\label{crucial}
Let $f\in W^{1,p}(\R^n)$ be not $0$ a.e. on $\R^n$. Then
\begin{equation}\label{eq:crucial}
\lt(\frac{\om_n}{V(K_{\lam,p}(f))}\rt)^{1/n} \lt(\int_{\R^n} h(K_{\lam,p}(f),\na f(y))^p dy\rt)^{1/p} \leq \mE_{\lam,p}(f),
\end{equation}
with equality if and only if $B_{\lam,p}(f)$ is an origin-centered ellipsoid.
\end{proposition}
\begin{proof}
It follows from the general $L_p$ Busemann-Petty centroid inequality (Theorem \ref{HS09}) and \eqref{eq:relationkey} that
\begin{equation}\label{eq:step}
V(\Klp) \geq ((n+p)\al_{n,p})^{n/p} V(\Blp)^{(n+p)/p} =((n+p)\al_{n,p})^{n/p} \lt(\frac1n\, \lt[\frac{\mE_{\lam,p}(f)}{c_{n,p}}\rt]^{-n}\rt)^{(n+p)/p},
\end{equation}
with equality if and only if $\Blp$ is an origin-centered ellipsoid.

Combining \eqref{eq:step} and \eqref{eq:intkey} finishes the proof.
\end{proof}

Finally, we will need the following elementary lemma.
\begin{lemma}
Let $K\in \mK_0^n$ such that $V(K) = \om_n$. Then for any function $f\in W^{1,p}(\R^n)$ and $\lam \in [0,1]$, we have
\begin{equation}\label{eq:last}
\int_{\R^n} h(K,-\na f^K(x))^p dx = \mE_{\lam,p}(f^\star)^p.
\end{equation}
\end{lemma}
\begin{proof}
Since $\|\cdot\|_K$ is a Lipschitz function on $\R^n$, it is differentiable at a.e. $x\in \R^n$. For such a point $x$ of differentiability, there exists a unique $x^*= \nabla(\|\cdot\|_K)(x)\in \R^n$ such that $h_K(x^*) = 1$, $\la x^*, x\ra = \|x\|_K$. Let 
$$\sigma_K(x) = \frac{x^*}{|x^*|}.$$
Since $\|\cdot\|_K$ is positive homogeneous of degree $1$, $\sigma_K(tx) = \sigma_K(x)$ with $t > 0$. Note that for $x\in \pa K$, $\si_K(x)$ is the outer unit normal vector at $x$. A simple computation shows that
$$\na f^K(x) = (f^*)'(\om_n \|x\|_K^n) n\om_n \|x\|_K^{n-1} x^*,$$
for a.e. $x\in \R^n$. Thus 
$$h(K,-\na f^K(x)) = (-f^*)'(\om_n \|x\|_K^n) n\om_n \|x\|_K^{n-1}.$$
Since $|x^*|^{-1} = h(K,\si_K(x))$, an application of the coarea formula together with \eqref{eq:normalized} yields
\begin{align*}
\int_{\R^n} h(K,-\na f^K(x))^p dx& = \int_{\R^n} (-(f^*)'(\om_n \|x\|_K^n) n\om_n \|x\|_K^{n-1})^p dx\\
&=\int_0^\infty (-(f^*)'(\om_n t^n)n\om_n)^p t^{(n-1)(p+1)} dt \int_{\pa K} h_K(\si_K(x^*)) d\mH^{n-1}(x)\\
&=(n\om_n)^p\int_0^\infty (-(f^*)'(\om_n t^n))^p t^{(n-1)(p+1)} dt \int_{S^{n-1}} h_K(u) dS(K,u)\\
&= (n\om_n)^{p+1}\int_0^\infty (-(f^*)'(\om_n t^n))^p t^{(n-1)(p+1)} dt\\
&= \int_{\R^n} |\na f^\star (x)|^p dx\\
&=\mE_{\lam,p}(f^\star)^p.
\end{align*}
This proves the lemma.
\end{proof}
We are now in a position to prove Theorem \ref{maintheorem}.

\begin{proof}[Proof of Theorem \ref{maintheorem}:] Let $f\in W^{1,p}(\R^n)$ be not $0$ a.e. on $\R^n$. It is easy to see that
\begin{equation}\label{eq:sym}
\mE_{\lam,p}(f) = \mE_{\lam,p}(-f).
\end{equation}
Let $B_{\lam,p}(-f)$ and $K_{\lam,p}(-f)$ be the convex bodies defined above. Denote by
$$K = \lt(\frac{\om_n}{V(K_{\lam,p}(-f))}\rt)^{1/n} K_{\lam,p}(-f).$$
Then $K\in \mK_0^n$ by Lemma \ref{key2} and $V(K) = \om_n$. Applying \eqref{eq:crucial} yields that
\begin{equation}\label{eq:trunggian}
\mE_{\lam,p}(-f) \geq \lt(\int_{\R^n} h(K, -\na f(x))^p dx\rt)^{1/p},
\end{equation}
with equality if and only if $B_{\lam,p}(-f)$ is an origin-centered ellipsoid. The P\'olya-Szeg\"o principle for convex symmetrization (Theorem \ref{Convexsymme}) implies that 
\begin{equation}\label{eq:PSconvex}
\int_{\R^n} h(K,-\na f(x))^p dx \geq \int_{\R^n} h(K, -\na f^K(x))^p dx = \mE_{\lam,p}(f^\star)^p.
\end{equation}
Combining \eqref{eq:sym}, \eqref{eq:trunggian} and \eqref{eq:PSconvex} proves the general affine P\'olya-Szeg\"o principle, that is, the inequality \eqref{eq:PSprin}.

Suppose that $f$ is a nonnegative function satisfying condition \eqref{eq:cond} and that 
$$\mE_{\lam,p}(f) = \mE_{\lam,p}(f^\star).$$
The equality holds in \eqref{eq:trunggian}. This implies that $B_{\lam,p}(-f)$ is an origin-centered ellipsoid, hence so is $K$ by its definition and \eqref{eq:relationkey}. Since $V(K) = \om_n$, there exists $T\in SL_n$ such that $K = TB_2^n$. We have
$\|x\|_K = |T^{-1}x|$, thus $f^K(x) = f^{\star}(T^{-1}x)$. The latter equality implies that 
$$V(\{x\, :\, |\na f^K(x)| =0\}\cap \{x\, :\, 0 < f^K(x) < {\rm ess} \sup f\}) =0.$$
Since equality also holds in \eqref{eq:PSconvex}, by Theorem \ref{Convexsymme}, there exists $x_0\in \R^n$ such that $f(x) = f^K(x+ x_0)$ for a.e. $x\in \R^n$. This finishes the proof of Theorem \ref{maintheorem} since $K$ is an origin-centered ellipsoid.
\end{proof}

We conclude this section with the remark that, when $p\in (1,n)$, the conclusions of Theorem \ref{maintheorem} also hold  for functions in the homogeneous Sobolev space $\dWp$ of real-valued functions on $\R^n$ which vanish at infinity such that their weak derivatives are in $L_p(\R^n)$. A function $f$ on $\R^n$ is said to vanish at infinity if for any $t > 0$, the Lesbegue measure of the set $\{x\in\R^n\, :\, |f(x)| > t\}$ is finite. Note that $W^{1,p}(\R^n) \subset \dWp$. And if $p\in (1,n)$, then $\dWp\subset L^{np/(n-p)}(\R^n)$ because of the Sobolev embedding theorem.


\section{Applications to affine Sobolev-type inequalities}
In this section, we use the general affine P\'olya-Szeg\"o principle to establish affine Sobolev-type inequalities related to the general affine $L_p$ energy. For example, we prove general affine Sobolev, general affine Morrey-Sobolev, general affine Gagliardo-Nirenberg and general affine logarithmic-Sobolev inequalities. These inequalities are  sharp and stronger than their classical Euclidean counterparts. Some of them are already known \cite{CLYZ09,HS09b,HSX12,LYZ02,LYZ06,TW15}. However, the characterization of their extremal functions were left open. Using the Brothers-Ziemer-type result established in Theorem \ref{maintheorem}, we can now classify all extremal functions for these inequalities. This classification of extremal functions seems to be new. For other asymmetric functional inequalities, see, e.g.,\cite{Ma,Ober,Schuster,Web}. 

\subsection{General affine $L_p$ Sobolev inequality}
The main result of this subsection is the following general affine $L_p$ Sobolev inequality.
\begin{corollary}\label{Sob}
Let $\lam \in [0,1]$, $p\in (1,n)$ and let $p^* = np/(n-p)$. Then for any function $f\in \dWp$, we have
\begin{equation}\label{eq:Sob}
S(n,p) \mE_{\lam,p}(f) \geq \|f\|_{p^*}
\end{equation}
where
$$S(n,p) = \pi^{-1/2} n^{-1/p} \lt(\frac{p-1}{n-p}\rt)^{1-1/p} \lt(\frac{\Gam(1 +n/2) \Gam(n)}{\Gam(n/p) \Gam(1+ n-n/p)}\rt)^{1/n}.$$
Moreover, equality holds in \eqref{eq:Sob} if and only if for a.e. $x\in \R^n$,
$$f(x) = \pm (a + |A(x-x_0)|^{p/(p-1)})^{1-n/p},$$
for some invertible linear map $A \in GL_n$, $a > 0$ and $x_0\in \R^n$.
\end{corollary}
Corollary \ref{Sob} includes the sharp affine $L_p$ Sobolev inequality of Lutwak et al. \cite{LYZ02} (corresponding to $\lam =1/2$), and the asymmetric affine $L_p$ Sobolev inequality of Haberl and Schuster \cite{HS09,HS09b} (corresponding to $\lam =0$). Inequality \eqref{eq:Sob} was recently proved by Wang in \cite{TW15} for functions $f\in W^{1,p}(\R^n)$ by exploiting the solution to the discrete functional $L_p$ Minkowski problem. In fact, Wang proved it for the dense subspace of $W^{1,p}(\R^n)$ of piecewise affine functions, and then obtained the inequality for functions in $W^{1,p}(\R^n)$ by a density argument.

\begin{proof}
Inequality \eqref{eq:Sob} is an easy consequence of the general affine P\'olya-Szeg\"o principle and the sharp Sobolev inequality on $\R^n$ with its equality conditions (see \cite{Ta76,CNV}). 

Suppose that equality holds in \eqref{eq:Sob} for a function $f$ which is not $0$ a.e. on $\R^n$. We first prove that $f$ does not change sign on $\R^n$. Indeed, writing $f =f_+ -f_-$, we have 
$$(1-\lam) (D_uf)_+^p + \lam (D_uf)_-^p = (1-\lam)(D_uf_+)_+^p + \lam (D_uf_+)_-^p + \lam (D_uf_-)_+^p + (1-\lam )(D_uf_-)_-^p,$$
hence
$$\|u\|_{p,\lam,f}^p = \|u\|_{p,\lam,f_+}^p + \|u\|_{p,1-\lam,f_-}^p.$$
The Minkowski inequality, and the general affine $L_p$ Sobolev inequality lead to
\begin{align}\label{eq:immestep}
S(n,p)\mE_{\lam,p}(f)^p&\geq S(n,p)\mE_{\lam,p}(f_+)^p + S(n,p)\mE_{1-\lam,p}(f_-)\notag\\
&\geq \lt(\int_{\R^n} f_+(x)^{p^*} dx\rt)^{p/p^*} + \lt(\int_{\R^n} f_-(x)^{p^*} dx\rt)^{p/p^*}\notag\\
&\geq \|f\|_{p^*}^p,
\end{align}
where the last inequality comes from the concavity of the function $t\to t^{p/p^*}$ on $(0,\infty)$. Since equality holds in \eqref{eq:Sob}, it also holds in \eqref{eq:immestep}. Thus either $\|f_+\|_{p^*}$ or $\|f_-\|_{p^*}$ must be zero because of the strict concavity of the function $t\to t^{p/p^*}$ on $(0,\infty)$. Hence $f$ does not change sign.

Without loss of generality we can assume that $f$ is nonnegative. Since equality holds in \eqref{eq:Sob}, we have equality in the general affine P\'olya-Szeg\"o principle and in the Sobolev inequality for $f^\star$. Thus $f^\star$ has the form as above, which ensures that the condition \eqref{eq:cond} in Theorem \ref{maintheorem} is satisfied. Theorem \ref{maintheorem} hence yields the existences of $x_0\in \R^n$ and an origin-centered ellipsoid $E$ such that $f(x) = f^E(x+ x_0)$ for a.e. $x\in \R^n$.
\end{proof}
\subsection{General affine Morrey-Sobolev inequality}
The classical Morrey-Sobolev inequality \cite{Ta93} states that if $f\in W^{1,p}(\R^n), p > 1$ such that $V({\rm supp} f) < \infty$. Then
$$\|f\|_\infty \leq b_{n,p} V({\rm supp} f)^{(p-n)/np}\|\na f\|_p,$$
where
$$b_{n,p} =n^{-1/p}\om_n^{-1/n} \lt(\frac{p-1}{p-n}\rt)^{(p-1)/p}.$$ 
Equality holds if and only if 
$$f(x) = \pm a (1 -|b(x-x_0)|^{(p-n)/(p-1)})_+,$$
for some $a, b> 0$ and $x_0\in \R^n$.

In this subsection, a general affine counterpart of this inequality is established.
\begin{corollary}\label{MS}
Let $f\in W^{1,p}(\R^n), p > 1$ such that $V({\rm supp} f) < \infty$, then
\begin{equation}\label{eq:MS}
\|f\|_\infty \leq n^{-1/p}\om_n^{-1/n} \lt(\frac{p-1}{p-n}\rt)^{(p-1)/p} V({\rm supp} f)^{(p-n)/np}\mE_{\lam,p}(f).
\end{equation}
Equality holds in \eqref{eq:MS} if and only if 
$$f(x) = \pm a (1 -|A(x-x_0)|^{(p-n)/(p-1)})_+,$$
for some $A\in GL_n$, $a> 0$ and $x_0\in \R^n$.
\end{corollary}
Inequality \eqref{eq:MS} was proved in \cite{CLYZ09} for $\lam =1/2$, in \cite{HS09} for $\lam =0$ and recently in \cite{TW15} in general for all $\lam \in [0,1]$. 
\begin{proof}
Inequality \eqref{eq:MS} is easily derived from the general affine P\'olya-Szeg\"o principle and the classical Morrey-Sobolev inequality above. 

Suppose that equality holds in \eqref{eq:MS} for a function $f$ which is not $0$ a.e. on $\R^n$. We first show that $f$ does not change sign on $\R^n$. Indeed, writing $f = f_+ -f_-$, we have 
$$\mE_{\lam,p}(f)^p \geq \mE_{\lam,p}(f_+)^p + \mE_{\lam,p}(f_-)^p.$$
If $f_+$ and $f_-$ are not $0$ a.e. on $\R^n$, then 
$$a_1 = V({\rm supp}f_+) > 0\quad \text{ and }\quad a_2 = V({\rm supp} f_-) > 0.$$ 
It follows from \eqref{eq:MS} and the H\"older inequality that
\begin{align*}
b_{n,p}^p\mE_{\lam,p}(f)^p &\geq a_1^{(n-p)/n} \|f_+\|_\infty^p + a_2^{(n-p)/n} \|f_-\|_\infty^p\\
&\geq (a_1+ a_2)^{(n-p)/n} (\|f_+\|_\infty^n + \|f_-\|_\infty^n)^{p/n}\\
&> V({\rm supp}f)^{(n-p)/n} \|f\|_\infty^p.
\end{align*}
This is impossible since equality holds in \eqref{eq:MS}, thus $f$ does not change sign on $\R^n$. Hence, without loss of generality, we may assume that $f$ is nonnegative. Since equality must hold in the classical Morrey-Sobolev inequality for $f^\star$, we have 
$$f^\star(x) = a (1 -|bx|^{(p-n)/(p-1)})_+,$$
for some $a, b > 0$. Hence, condition \eqref{eq:cond} of Theorem \ref{maintheorem} is satisfied. Thus, Theorem \ref{maintheorem} implies that there exist $x_0 \in \R^n$ and an origin-centered ellipsoid $E$ such that $f(x) = f^E(x+ x_0)$ for a.e. $x\in \R^n$.  
\end{proof}
\subsection{General affine Gagliardo-Nirenberg inequality}
The main result of this subsection is the following general affine Gagliardo-Nirenberg inequality.
\begin{corollary}\label{GN}
Let $p\in (1,n)$ and $\al \in (0, n/(n-p))$, $\al\not=1$.
\begin{description}
\item (i) If $\al > 1$, then there exists a constant $G(n,\al,p)$  such that for any function $f\in \dWp\cap L^{\al(p-1)+1}(\R^n)$, we have
\begin{equation}\label{eq:big1}
\|f\|_{\al p} \leq G(n,\al,p) \mE_{\lam,p}(f)^\theta \|f\|_{\al(p-1)+1}^{1-\theta},
\end{equation}
where
$$\theta = \frac{n(\al -1)}{\al(np - (\al p+ 1-\al)(n-p))}$$
and 
$$G(n,\al,p) = \lt(\frac{y(\al-1)^p}{\pi^{p/2}q^{p-1}n}\rt)^{\theta/p} \lt(\frac{qy-n}{qy}\rt)^{1/\al p} \lt(\frac{\Gam(y)\Gam(1 +n/2)}{\Gam(y -n/q)\Gam(1 + n/q)}\rt)^{\theta/n},$$
with $y = (\al p-\al+1)/(\al-1)$ and $q =p/(p-1)$. Moreover, equality holds in \eqref{eq:big1} if and only if there exist $x_0\in \R^n$, $a >0$ and $A\in GL_n$ such that 
$$f(x) = \pm(a + |A(x-x_0)|^{p/(p-1)})^{-1/(\al -1)},$$
for a.e. $x\in \R^n$.
\item (ii) If $\al\in (0,1)$, then there exists a constant $G(n,\al,p)$  such that for any function $f\in \dWp\cap L^{\al p}(\R^n)$, it holds
\begin{equation}\label{eq:small1}
\|f\|_{\al(p-1) +1} \leq G(n,\al,p) \mE_{\lam,p}(f)^\theta \|f\|_{\al p}^{1-\theta},
\end{equation}
where
$$\theta = \frac{n(1-\al)}{(\al p +1-\al)(n-\al(n-p))}$$
and
$$G(n,\al,p) = \lt(\frac{y(1-\al)^p}{\pi^{p/2}q^{p-1} n}\rt)^{\theta/p} \lt(\frac{qy}{qy +n}\rt)^{(1-\theta)/\al p} \lt(\frac{\Gam(y + 1 + n/q)\Gam(1+n/2)}{\Gam(1 +z)\Gam(1 + n/q)}\rt)^{\theta/n}$$
with $y = (\al p -\al +1)/(1-\al)$ and $q = p/(p-1)$. Moreover, equality holds in \eqref{eq:small1} if and only if there exists $x_0\in \R^n$, $a > 0$ and $A\in GL_n$ such that 
$$f(x) = \pm(a - |A(x-x_0)|^{p/(p-1)})_+^{1/(1-\al)},$$
for a.e. $x\in \R^n$.
\end{description} 
\end{corollary}
Corollary \ref{GN} was proved in \cite{Zhai11} for $\lam =1/2$, and recently in \cite{HJM15} by a different proof. The case $\lam=1/2$ was proved in \cite{HS09}.
\begin{proof}
\emph{Proof of part $(i)$:} Inequality \eqref{eq:big1} is proved by combining the general affine P\'olya-Szeg\"o principle and the classical Gagliardo-Nirenberg inequality \cite{CNV,DD1,DD2}.

Suppose that equality holds in \eqref{eq:big1} for a function $f$ which is not $0$ a.e. on $\R^n$. We first show that $f$ does not change sign on $\R^n$. Indeed, writing $f=f_+ -f_-$, we have
$$\mE_{\lam,p}(f)^p \geq \mE_{\lam,p}(f_+)^p + \mE_{\lam,p}(f_-)^p.$$
Denote
$$\ga = (1-\theta) \frac{\al p}{\al(p-1) + 1} = \frac{np - \al p(n-p)}{np - (\al p -\al +1)(n-p)}.$$
Then
$$1-\ga = \frac{(\al -1)(n-p)}{np - (\al p -\al +1)(n-p)} =\frac{n-p}n \al\theta \in (0, \al \theta).$$
Using \eqref{eq:big1} and the H\"older inequality, we get
\begin{align*}
\|f\|_{\al p}^{\al p} & = \|f_+\|_{\al p}^{\al p} + \|f_-\|_{\al p}^{\al p}\\
&\leq G(n,\al,p)^{\al p} \lt(\mE_{\lam,p}(f_+)^{\al p\theta} \|f_+\|_{\al p-\al +1}^{(1-\theta)\al p} + \mE_{\lam,p}(f_-)^{\al p\theta} \|f_-\|_{\al p-\al +1}^{(1-\theta)\al p}\rt)\\
&\leq G(n,\al,p)^{\al p}(\mE_{\lam, p}(f_+)^{\al\theta p/(1-\ga)} + \mE_{\lam, p}(f_-)^{\al\theta p/(1-\ga)})^{1-\ga}\|f\|_{\al p -\al +1}^{(1-\theta)\al p}\\
&\leq G(n,\al,p)^{\al p}(\mE_{\lam, p}(f_+)^{p} + \mE_{\lam, p}(f_-)^{p})^{\al \theta} \|f\|_{\al p -\al +1}^{(1-\theta)\al p}\\
&\leq G(n,\al,p)^{\al p} \mE_{\lam,p}(f)^{\al p\theta} \|f\|_{\al p -\al +1}^{(1-\theta)\al p},
\end{align*}
here we used the inequality $a^{\al\theta/(1-\ga)} + b^{\al\theta/(1-\ga)} \leq (a +b)^{\al\theta/(1-\ga)}$ for any $a, b\geq 0$ since $\al \theta /(1-\ga) =n/(n-p) > 1$. This inequality is strict unless $a$ or $b$ are equal to $0$. Hence equality in \eqref{eq:big1} implies that either $\mE_{\lam,p}(f_+)$ or $\mE_{\lam,p}(f_-)$ is equal to $0$. Thus, $f$ does not change sign on $\R^n$. The rest of the proof is similar to the one of Corollary \ref{Sob}, using the result of Del Pino and Dolbeault \cite{DD1,DD2} about the extremal functions of the classical Gagaliardo-Nirenberg inequality.

\emph{Proof of part $(ii)$:} Part $(ii)$ is proved in the same way.
\end{proof}

\subsection{General affine logarithmic-Sobolev inequality}
The classical sharp $L_p$ logarithmic-Sobolev inequality was proved by Del Pino and Dolbeault \cite{DD1,DD2} for $p\in (1,n)$ and was extended to all $p> 1$ by Gentil \cite{Gen}. This inequality states that for $n\geq 1$, $p >1$ and any functions $f\in W^{1,p}(\R^n)$ such that $\int_{\R^n} |f(x)|^p dx =1$, we have
\begin{equation}\label{eq:logSob}
\Ent(|f|^p) = \int_{\R^n} |f|^p \ln (|f|^p) dx \leq \frac np \ln\lt(\mL_{p} \int_{\R^n} |\na f|^p dx\rt),
\end{equation}
where 
$$\mL_p = \pi^{-p/2}\frac pn \lt(\frac{p-1}e\rt)^{p-1} \lt(\frac{\Gam(1 +n/2)}{\Gam(1 + n(p-1)/p)}\rt)^{p/n}.$$
Moreover, equality holds in \eqref{eq:logSob} if and only if for some $\si > 0$ and $x_0\in \R^n$,
$$f(x) = \pi^{-n/2p}(\si p^{(p-1)/p})^{n/p}\lt(\frac{\Gam(1 +n/2)}{\Gam(1 + n(p-1)/p)}\rt)^{1/p} e^{- |\si(x-x_0)|^{p/(p-1)}}.$$ 
The case $p=2$ of \eqref{eq:logSob} is very interesting since it is equivalent to Gross's logarithmic-Sobolev inequality for Gaussian measure \cite{Gross} which has many important applications in analysis, probability and quantum field theory. 

The affine counterpart of \eqref{eq:logSob} was proved by Zhai \cite{Zhai11} for $\lam =1/2$ and $p\in (1,n)$ by exploiting the affine P\'olya-Szeg\"o principle and the classical logarithmic-Sobolev inequality \eqref{eq:logSob}. Another proof of this inequality can be found in \cite{HJM15} for all $p > 1$. The asymmetric affine logarithmic-Sobolev inequality (with $\lam =0$) was recently established in \cite{HS09}. In this subsection, we prove a general affine logarithmic-Sobolev inequality for all $\lam\in[0,1]$ and $p >1$. Moreover, all extremal functions are characterized.

\begin{corollary}\label{affinLS}
Let $n\geq 1$ and $p >1$, then for any functions $f\in W^{1,p}(\R^n)$ such that $\int_{\R^n} |f(x)|^p dx =1$, we have
\begin{equation}\label{eq:affineLS}
\Ent(|f|^p) \leq \frac np \ln\lt(\mL_p \mE_{\lam,p}(f)^p\rt),
\end{equation}
where $\mL_p$ is as above. Equality holds in \eqref{eq:affineLS} if and only if for some $\si > 0$ and $A\in SL_n$ (the set of $n\times n$ matrices of determinant $1$), we have
$$f(x) = \pi^{-n/2p} (\si p^{(p-1)/p})^{n/p}\lt(\frac{\Gam(1 +n/2)}{\Gam(1 + n(p-1)/p)}\rt)^{1/p} e^{- |\si A(x-x_0)|^{p/(p-1)}}.$$
\end{corollary}
\begin{proof}
Inequality \eqref{eq:affineLS} is proved by combining the general affine P\'olya-Szeg\"o principle and the classical logarithmic-Sobolev inequality \eqref{eq:logSob}. 

Suppose that equality holds in \eqref{eq:affineLS} for a function $f$ which is not $0$ a.e. on $\R^n$. We will prove that the function $f$ does not change sign on $\R^n$. Indeed, writing again $f = f_+ - f_-$, we show that either $f_+$ or $f_-$ is zero a.e. on $\R^n$. Indeed, if this does not hold, define $a_1 = \|f_+\|_p, a_2 =\|f_-\|_p$, $f_1 = f_+/a_1$ and $f_2 = f_-/a_2$. Then by using the strict convexity of the function $e^t$ and function $t \ln t$, and the inequality \eqref{eq:affineLS}, we have (note that $a_1^p + a_2^p =1$)
\begin{align*}
\mL_p\Om_{\lam,p}(f)^p& \geq a_1^p \mL_p\mE_{\lam,p}(f_1)^p + a_2^p \mL_p\mE_{\lam,p}(f_2)^p\\
&\geq  \lt(a_1^p e^{p \Ent(f_1^p)/n} + a_2^p e^{p \Ent(f_2^p)/n}\rt)\\
&\geq e^{(a_1^p \Ent(f_1^p) + a_2^p \Ent(f_2^p))p/n} \\
&\geq e^{\Ent(|f|^p) p/n},
\end{align*}
with equality if and only if $\Ent(f_1^p) = \Ent(f_2^p)$ and $f_1 =f_2$ a.e. on $\R^n$. This implies that $f =0$ a.e in $\R^n$ which contradicts the assumption on $f$. Hence $f$ does not change sign on $\R^n$. The rest of the proof is similar with the one of Corollary \ref{Sob}.
\end{proof}


\section{Stability of the affine Sobolev inequality}
As in Section \S3, for $f\in BV(\R^n)$ which is not $0$ a.e. on $\R^n$ and $x\in \R^n$, we define
$$\|x\|_{1,f} = \int_{\R^n} |\la x, \si_f(y)\ra| d(|Df|)(y),$$
and
$$B_1(f) = \{x\in \R^n\, :\, \|x\|_{1,f} \leq 1\}.$$

\begin{lemma}
Suppose that $f \in BV(\R^n)$ is not $0$ a.e. on $\R^n$, then $B_1(f)$ is an origin-symmetric convex body in $\mK_0^n$.
\end{lemma}
\begin{proof}
It is easy to see that $B_1(f)$ is convex and symmetric. We next show that $B_1(f)$ is compact and contains the origin in its interior. To do this, it is enough to show that there exist constants $C,c > 0$ such that
$$c \leq \|u\|_{1,f} \leq C,\quad \forall\, u\in S^{n-1}.$$
Since $|\la u, \si_f(y)\ra|\leq 1$ a.e. with respect to $|Df|$, we can choose $C = |Df|(\R^n)$ in the previous inequality. For the existence of $c > 0$, it suffices to show that 
\begin{equation}\label{eq:e1}
\int_{\R^n} |\la u, \si_f(y)\ra| d(|Df|)(y) > 0,\quad\forall\, u \in \R^n.
\end{equation}
Otherwise, if there exists $u \in S^{n-1}$ such that $\int_{\R^n} |\la u, \si_f(y)\ra| d(|Df|)(y) =0$, then we must have $\la u, \si_f(y)\ra =0$ a.e. with respect to $|Df|$. For any function $\vphi\in C_0^\infty(\R^n)$, we have
$$0 = \int_{\R^n} \la u,\si_f(y) \ra \vphi(y) d(|Df|)(y) = -\int_{\R^n} f(y) D_u\vphi(y) dy.$$
Thus $D_uf =0$ in the distributional sense. Denote $f_\de = f\star \psi_\de$, $\de > 0$ with $\psi_\de$ defined as above. Then $f_\de \in C^\infty(\R^n)$ and $f_\de$ converges to $f$ in $ L_{n/(n-1)}(\R^n)$ as $\de \to 0$.

We have $D_u f_\de = D_u(f\star \psi_\de) = (D_u f)\star \psi_\de$, that is, $D_u f_\de(x) = 0$ for all $x\in \R^n$. Thus $f_\de$ is constant in the direction $u$, but $f_\de\in L_{n/(n-1)}(\R^n)$, hence $f_\de(x) =0$ for all $x\in \R^n$. Taking $\de\to 0$ implies that $f(x) = 0$ for a.e. $x\in \R^n$ with respect to Lebesgue measure. This contradicts  the assumption on $f$. Hence \eqref{eq:e1} holds, and our proof is complete.
\end{proof}
Let $K_1(f)$ denote the convex body whose support function is
$$h(K_1(f),u) = \int_{S^{n-1}} \|\xi\|_{1,f}^{-n-1} |\la u, \xi\ra| d\xi.$$
Using the same proof as in Section \S3, we have the following result.
\begin{lemma}\label{case1}
Suppose that $f \in BV(\R^n)$ is not $0$ a.e. with respect to Lesbegue measure. Then
$$K_1(f) = (n+1) \al_{n,1} V(B_1(f)) \Gam_1B_1(f),$$
$$\int_{\R^n} h(K_1(f),\si_f(y)) d(|Df|)(y) = \lt(\frac{\mE_1(f)}{c_{n,1}}\rt)^{-n},$$
$$\mE_1(f) \geq \lt(\frac{\om_n}{V(K_1(f))}\rt)^{1/n} \int_{\R^n} h(K_1(f),\si_f(y)) d(|Df|)(y),$$
with equality if and only if $B_1(f)$ is an origin-centered ellipsoid, and
$$\int_{\R^n} h(K, \si_{f^K}(y)) d(|Df^K|)(y) = \mE_1(f^\star),$$
for any origin-symmetric convex bodies $K$ such that $V(K) = \om_n$.
\end{lemma}

For origin-symmetric convex bodies $K,L$ in $\R^n$, the Banach-Mazur distance between $K$ and $L$ is defined by
$$\de_{BM}(K,L) = \min\{\de \geq 0\, :\, K \subset \Phi(L) \subset e^\de K \quad \text{for}\quad \Phi \in GL_n\}.$$  
For a convex body $K$ in $\R^n$, its projection body $\Pi_1K$ is defined by
$$h(\Pi_1K,u)=\frac{V_{n-1}({\rm Pr}_u( K))}{\om_{n-1}} = \frac1{2\om_{n-1}} \int_{S^{n-1}} |\la u,v\ra| \, dS(K,v),$$
where $V_{n-1}$ denotes the $(n-1)$ dimensional volume and ${\rm Pr}_u (K)$ is the projection of $K$ on the hyperplane $u^\bot$. Denote by $\Pi_1^*K$ the polar of $\Pi_1 K$. The Petty projection inequality is one of the classical affine isoperimetric inequalities \cite{Pet71}. It states that for any convex body $K$ in $\R^n$,
$$V(\Pi_1^*K) V(K)^{n-1} \leq \om_n^n,$$
with equality if and only if $K$ is an ellipsoid. A stability estimate for the Petty projection inequality recently proved by B\"or\"oczky \cite{KB13} reads as follows: For any origin-symmetric convex body $K$ in $\R^n$,
\begin{equation}\label{eq:StabPetty}
V(\Pi_1^*K) V(K)^{n-1}\leq (1 -\ga \de_{BM}(K,B_2^n)^{cn}) \om_n^n,
\end{equation}
where $c = 1680$ and $\ga > 0$ depends only on $n$.

In order to establish Theorem \ref{maintheo1}, we will prove a stability estimate for the Busemann-Petty centroid inequality. In fact, this stability estimate is derived from \eqref{eq:StabPetty} by using a class reduction technique introduced by Lutwak \cite{Lut86} and the improved dual mixed volume inequality \eqref{eq:improved}. For origin-symmetric convex bodies $K,L$ in $\R^n$, let us denote
$$A(K,L) = \frac{V(K \De (a A))}{V(K)}\quad\text{with}\quad a = \lt(\frac{V(K)}{V(L)}\rt)^{1/n}.$$
It is well known that
\begin{equation}\label{eq:comparedistance}
c_1(n) \de_{BM}(K,L)^n \leq \inf\{A(K,\psi(L))\,:\, \psi \in GL_n\} \leq c_2(n) \de_{BM}(K,L),
\end{equation}
where $c_1(n),c_2(n)$ depend only on $n$ (see \cite[Section $5$]{BH14} and \cite[Section $8$]{BH15}).

Our stability estimate for the Busemann-Petty centroid inequality can be now stated as follows.
\begin{proposition}\label{Stabcentroid}
Let $K$ be an origin-symmetric convex body in $\R^n$. Then
\begin{equation}\label{eq:Stabcentro}
\frac{V(\Gam_1K)}{V(K)} \geq 1 + C(n) \de_{BM}(K,B_2^n)^{cn},
\end{equation}
where $C(n)$ depends only on $n$ and $c =1680$. Consequently, we have
$$\frac{V(\Gam_1K)}{V(K)} \geq 1 + \t{C}(n) \lt(\inf\{A(K,E)\,:\, E\in \mE\}\rt)^{cn},$$
where $\t{C}(n)$ depends only on $n$.
\end{proposition} 
\begin{proof}
We start with the following equality given in \cite{Lut90} for any star body $K$ and convex body $L$,
$$V_1(L, \Gam_1K) = \frac{\om_n}{V(K)} \t{V}_{-1}(K, \Pi_1^*L).$$
Taking $L = \Gam_1K$ leads to
$$V(\Gam_1K) = \frac{\om_n}{V(K)} \t{V}_{-1}(K, \Pi_1^*\Gam_1K).$$
In order to simplify the notation, let us denote
$$\de = \frac{V(\Gam_1K)}{V(K)}- 1,\quad \de_1 =\de_{BM}(\Gam_1 K,B_2^n), \quad\text{and}\quad \eta = A(K, \Pi_1^*\Gam_1K)$$
According to \eqref{eq:improved} with $p=1$, we have
$$V(\Gam_1K) \geq \om_n V(K)^{1/n} V(\Pi_1^*\Gam_1K)^{-1/n} \lt(1 + \frac{\eta^2}{8n} \rt),$$
or equivalently
\begin{equation}\label{eq:abcd}
\frac{V(\Gam_1K)}{V(K)} \geq \om_n^{n} V(\Gam_1K)^{1-n} V(\Pi_1^*\Gam_1K)^{-1} \lt(1 + \frac{\eta^2}{8n}\rt)^n.
\end{equation}
Plugging \eqref{eq:StabPetty} into \eqref{eq:abcd} implies that
\begin{align*}
1+ \de &\geq (1- \ga \de_1^{cn})^{-1} \lt(1 + \frac{\eta^2}{8n}\rt)^n \geq (1+ \ga \de_1^{cn})\lt(1 + \frac{\eta^2}{8} \rt)\geq 1 + \ga \de_1^{cn} + \frac{\eta^2}8.
\end{align*}
According to \eqref{eq:comparedistance}, we have $\eta \geq c_1(n) \de_{BM}(K,\Pi_1^*\Gam_1K)^n$. Thus, there exists $c(n) > 0$ depending only on $n$ such that  
\begin{equation}\label{eq:1}
\de \geq c(n)(\de_1+ \de_{BM}(K,\Pi_1^*\Gam_1K))^{cn}.
\end{equation}

Let $E$ be the origin-centered ellipsoid $E$ such that $E\subset \Gam_1K \subset e^{\de_1} E$. Then 
$$\Pi_1 E \subset \Pi_1 \Gam_1K \subset e^{(n-1)\de_1} \Pi_1E,$$
or equivalently 
$$e^{-(n-1)\de_1} \Pi_1^*E \subset \Pi_1^*\Gam_1K \subset \Pi_1^*E.$$
Thus, we have
\begin{equation}\label{eq:2}
\de_{BM}(\Pi_1^*\Gam_1K, B_2^n) \leq (n-1) \de_1.
\end{equation}

Plugging \eqref{eq:2} into \eqref{eq:1} and using the triangle inequality for the Banach-Mazur distance yields the desired inequality \eqref{eq:Stabcentro}.
\end{proof}

We are now ready to prove Theorem \ref{maintheo1}.

\begin{proof}[Proof of Theorem \ref{maintheo1}:] By homogeneity, we can assume that $\|f\|_{n'} =1$. Denote 
$$K = \lt(\frac{\om_n}{V(K_1(f))}\rt)^{1/n} K_1(f).$$
Then $K$ is an origin-centered convex body and $V(K) =\om_n$. Using Lemma \ref{case1}, Proposition \ref{Stabcentroid} and the affine Sobolev inequality, we have
\begin{align}\label{eq:Fi}
\de_a(f)& = \frac{\mE_1(f)}{n\om_n^{1/n}} \lt(1 - \lt(\frac{V(B_1(f))}{V(\Gam_1B_1(f))}\rt)^{1/n}\rt) + \frac{\int_{\R^n} h_K(\si_f) d(|Df|)}{n\om_n^{1/n}} -1\notag\\
&\geq 1 - \frac{1}{(1+C(n) \de_{BM}(B_1(f),B_2^n)^{cn})^{1/n}} + \frac{\int_{\R^n} h_K(\si_f) d(|Df|)}{n\om_n^{1/n}} -1.
\end{align}
In particular, we have
$$\de_a(f) \geq 1 - \frac{1}{(1+C(n) \de_{BM}(B_1(f),B_2^n)^{cn})^{1/n}},$$
or equivalently,
\begin{equation}\label{eq:123}
\de_{BM}(B_1(f),B_2^n) \leq \lt(\frac{(1-\de_a(f))^{-n} -1}{C(n)}\rt)^{1/cn}.
\end{equation}
Since $\de_{BM}(B_1(f),B_2^n)\leq (\ln n)/2$ by John's theorem, \eqref{eq:123} implies the existence of a constant $C_1(n)$ depending only on $n$ such that
$$\de_{BM}(B_1(f),B_2^n) \leq C_2(n) \de_a(f)^{1/cn}.$$
From the definition of the $L_1$ centroid body, it is easy to prove that
\begin{equation}\label{eq:hoang}
\de_{BM}(K,B_2^n) =\de_{BM}(\Gam_1B_1(f),B_2^n) \leq C_3(n) \de_a(f)^{1/cn},
\end{equation}
where $C_3(n)$ depends only on $n$.

If $\de$ denotes $\de_{BM}(K,B_2^n)$, then there exists an origin-centered ellipsoid such that 
$$E\subset K\subset e^{\de}E.$$
Taking $\psi \in SL_n$ such that $(\om_n/V(E))^{1/n} E = \psi (B_2^n)$, we have
$$e^{-\de} \psi (B_2^n) \subset K \subset e^{\de} \psi (B_2^n).$$
Consequently, since $\de \leq (\ln n)/2$ by John's theorem, we have
\begin{equation}\label{eq:thom}
V(K \De \psi (B_2^n)) \leq \om_n (e^{n\de} -e^{-n\de}) \leq C_4(n) \om_n \de,
\end{equation}
where $C_4(n)$ depends only on $n$.
 
The quantitative anisotropic Sobolev inequality (see \cite{FMP13}) implies the existence of $a\not=0$ and $x_0\in \R^n$ such that 
$$\frac{\int_{\R^n} h_K(\si_f) d(|Df|)}{n\om_n^{1/n}} -1 \geq \frac1{C_5(n)} \lt(\int_{\R^n} |f - a \chi_{x_0 +ar(a)K}|^{n'} dx\rt)^2,$$
where $r(a) = \om_n^{-1/n}|a|^{-n'}$ is chosen such that $\int_{\R^n} |a\chi_{x_0+ ar(a)K}|^{n'} dx =1$, and $C_4(n)$ depends only on $n$. Consequently, we have
\begin{equation}\label{eq:ngan}
\lt(\int_{\R^n} |f - a \chi_{x_0 +ar(a)K}|^{n'} dx\rt)^2 \leq C_5(n) \de_a(f).
\end{equation}

From \eqref{eq:hoang}, \eqref{eq:thom} and \eqref{eq:ngan}, we have
\begin{align*}
\int_{\R^n}& |f - a \chi_{x_0 +ar(a)\psi B_2^n}|^{n'} dx\\ 
&\leq 2^{1-n'}\lt(\int_{\R^n} |f - a \chi_{x_0 +ar(a)K}|^{n'} dx + \int_{\R^n} | a \chi_{x_0 +ar(a)K}- a \chi_{x_0 +ar(a)\psi (B_2^n)}|^{n'} dx\rt)\\
&\leq 2^{1-n'}\lt((C_5(n) \de_a(f))^{1/2} + |a|^{n'+n} r(a)^n V(K \De \psi (B_2^n))\rt)\\
&\leq 2^{1-n'}\lt((C_5(n) \de_a(f))^{1/2} + C_4(n) \de\rt)\\
&\leq 2^{1-n'}\lt((C_5(n) \de_a(f))^{1/2} + C_4(n) C_3(n) \de_a(f)^{1/cn}\rt)\\
&\leq \al(n) \de_a(f)^{1/cn}.
\end{align*}
This completes the proof of Theorem \ref{maintheo1}.
\end{proof}

We finally remark that Theorem \ref{maintheo1} immediately implies a stability estimate for  the affine $L_1$ logarithmic--Sobolev inequality. Let us recall that if $n\geq 2$, then for every $f\in BV(\R^n)$ which is not $0$ a.e. on $\R^n$, we have 
\begin{equation}\label{eq:affineLS1}
\int_{\R^n} \frac{|f(x)|}{\|f\|_1} \ln\lt(\frac{|f(x)|}{\|f\|_1}\rt) dx \leq n \ln \lt(\frac{\mE_1(f)}{n \om_n^{1/n} \|f\|_1}\rt).
\end{equation}
Inequality \eqref{eq:affineLS1} follows immediately from the affine Sobolev inequality \eqref{eq:ZhangWang} by the following argument. Obviously, we can assume $\|f\|_1 =1$ by homogeneity. Thus, by Jensen's inequality and the affine Sobolev inequality \eqref{eq:ZhangWang}, we have 
\begin{align*}
\int_{\R^n} |f(x)| \ln (|f(x)|) dx &= \frac1{n'-1} \int_{\R^n} |f(x)| \ln \lt(|f(x)|^{n'-1}\rt) dx\\
&\leq \frac1{n'-1} \ln\lt(\int_{\R^n} |f(x)|^{n'} dx\rt)\\
&\leq n \ln \lt(\frac{\mE_1(f)}{n \om_n^{1/n}}\rt),
\end{align*}
which is \eqref{eq:affineLS1}. Moreover, we have
\[
\frac{\|f\|_1}{\|f\|_{n'}} \exp\lt(\frac 1n\int_{\R^n} \frac{|f(x)|}{\|f\|_1} \ln\lt(\frac{|f(x)|}{\|f\|_1}\rt) dx\rt) \leq 1.
\]
Hence, if we set 
\[
\de_{aLS}(f) = \frac{\mE_1(f)}{n \om_n^{1/n}\|f\|_{n'}}-\frac{\|f\|_1}{\|f\|_{n'}} \exp\lt(\frac 1n\int_{\R^n} \frac{|f(x)|}{\|f\|_1} \ln\lt(\frac{|f(x)|}{\|f\|_1}\rt) dx\rt).
\]
We then have
\[
\de_{aLS}(f) \geq \de_a(f).
\]
As a corollary of the previous inequality and Theorem \ref{maintheo1}, we deduce the following stability estimate for affine logarithmic--Sobolev inequality \eqref{eq:affineLS1}.
\begin{corollary}\label{cor}
For $n\geq 2$ and $f \in BV(\R^n)$, we have 
\[
d_a(f,\mM) \leq \alpha(n) \,\de_{aLS}(f)^{1/cn},
\]
with $c= 1680$ and $\alpha(n)$ is the constant given in Theorem \ref{maintheo1}.
\end{corollary} 
An Euclidean counterpart of Corollary \ref{cor} can be found in \cite{FMP13}.
\section*{Acknowledgments}
The author is grateful to thank the anonymous referee for her/his great help in improving the English presentation of this paper and for useful comments that improved the exposition of this paper. This work was supported by a grant from the European Research Council (grant number 305629).


\begin{thebibliography}{99}
\bibitem{AFTL97}
A. Alvino, V. Ferone, G. Trombetti, and P. L. Lions, \emph{Convex symmetrization and applications\text}, Ann. Inst. Henri Poincar\'e. Analyse Nonlin\'eaire, {\bf 14} (1997) 275-293.


\bibitem{BB10}
K. Ball, and K. J. B\"or\"oczky, \emph{Stability of the Pr\'ekopa-Leindler inequality\text}, Mathematika, {\bf 56} (2010) 339-356.

\bibitem{BB11}
K. Ball, and K. J. B\"or\"oczky, \emph{Stability of some versions of the Pr\'ekopa-Leindler inequality\text}, Monatsh. Math., {\bf 163} (2011) 1-14.

\bibitem{KB10}
K. J. B\"or\"oczky, \emph{Stability of the Blaschke-Santal\'o and the affine isoperimetric inequalities\text}, Adv. in Math., {\bf 225} (2010) 1914-1928.

\bibitem{KB13}
K. J. B\"or\"oczky, \emph{Stronger versions of the Orlicz-Petty projection inequality\text}, J. Diff. Geom., {\bf 95} (2013) 215-247.

\bibitem{BH14}
K. J. B\"or\"oczky, and M. Henk, \emph{Cone volume measure and stability\text}, preprint, arXiv:1407.7272.

\bibitem{BH15}
K. J. B\"or\"oczky, and D. Hug, \emph{Isotropic measures and stronger forms of the reverse isoperimetric inequality\text}, Trans. Amer. Math. Soc., (to appear).


\bibitem{BZ88}
J. E. Brothers, and W. P. Ziemer, \emph{Minimal rearrangements of Sobolev functions\text}, J. Reine Angew. Math., {bf 384} (1988) 153-179.

\bibitem{CG02}
S. Campi, and P. Gronchi, \emph{The $L^p$ Busemann-Petty centroid inequality\text}, Adv. Math., {\bf 167} (2002) 128-141.

\bibitem{Ci06}
A. Cianchi, \emph{A quantitative Sobolev inequality in BV\text}, J. Funct. Anal., {\bf 237} (2006) 466-481.


\bibitem{CEFT08}
A. Cianchi, L. Esposito, N. Fusco, and C. Trombetti, \emph{A quantitative P\'olya-Szeg\"o principle\text}, J. Reine Angew. Math., {\bf 614} (2008) 153-189.




\bibitem{CFMP09}
A. Cianchi, N. Fusco, F. Maggi, and A. Pratelli, \emph{The sharp Sobolev inequality in quantitative form\text}, J. Eur. Math. Soc., {\bf 11} (5) (2009) 1105-1139.

\bibitem{CLYZ09}
A. Cianchi, E. Lutwak, D. Yang, and G. Zhang,\emph{Affine Moser-Trudinger and Morrey-Sobolev inequalities\text}, Calc. Var. Partial Differential Equations, {\bf 36} (2009) 419-436.

\bibitem{CNV}
D. Cordero-Erausquin, B. Nazaret, and C. Villani, \emph{A mass-transportation approach to sharp Sobolev and Gagliardo-Nirenberg inequalities\text}, Adv. Math., {\bf 182} (2004) 307-332.

\bibitem{DD1}
M. Del Pino, and J. Dolbeault, \emph{Best constants for Gagliardo-Nirenberg inequalities and applications to nonlinear diffusions\text}, J. Math. Pures Appl., {\bf 81} (9) (2002) 847-875.

\bibitem{DD2}
M. Del Pino, and J. Dolbeault, \emph{The optimal Euclidean Lp-Sobolev logarithmic inequality\text}, J. Funct. Anal., {\bf 197} (01) (2003) 151-161.

\bibitem{EG}
L. C. Evans, and R. F. Gariepy, \emph{Measure theory and fine properties of functions\text}, CRC Press, Boca Raton, FL, 1992. 

\bibitem{ET04}
L. Esposito, and C. Trombetti, \emph{Convex symmetrization and P\'olya-Szeg\"o inequality\text}, Nonlinear Anal., {\bf 56} (2004) 43-62.

\bibitem{ER09}
L. Esposito, and P. Ronca, \emph{Quantitative P\'olya-Szeg\"o principle for convex symmetrization\text}, Manuscripta Math., {\bf 130} (2009) 433-459.

\bibitem{FV04}
A. Ferone, and R. Volpicelli, \emph{Convex rearrangement: equality cases in the P\'olya-Szeg\"o inequality\text}, Calc. Var. Partial Differential Equations, {\bf 21} (2004) 259-272.

\bibitem{FiMP09}
A. Figalli, F. Maggi, and A. Pratelli, \emph{A refined Brunn-Minkowski inequality for convex sets\text}, Ann. Inst. H. Poincar\'e Anal. Non Lin\'eaire, {\bf 26} (2009) 2511-2519.

\bibitem{FMP10}
A. Figalli, F. Maggi, and A. Pratelli, \emph{A mass transportation approach to quantitative isoperimetric inequalities\text}, Invent. Math., {\bf 182} (2010) 167-211.

\bibitem{FMP13}
A. Figalli, F. Maggi, and A. Pratelli, \emph{Sharp stability theorems for the anisotropic Sobolev and log-Sobolev inequalities on functions of bounded variations\text}, Adv. Math., {\bf 242} (2013) 80-101.

\bibitem{FJ14}
A. Figalli, and D. Jerison, \emph{Quantitative stability for the Brunn-Minkowski inequality\text}, preprint, arXiv:1502.06513v1.

\bibitem{FMP07}
N. Fusco, F. Maggi, and A. Pratelli, \emph{The sharp quantitative Sobolev inequality forfunctions of bounded variation\text}, J. Funct. Anal., {\bf 244} (2007) 315-341.

\bibitem{FMP08}
N. Fusco, F. Maggi, and A. Pratelli, \emph{The sharp quantitative isoperimetric inequality\text}, Ann. of Math. (2), {\bf 168} (2008) 941-980.

\bibitem{Firey62}
W. J. Firey, \emph{$p-$means of convex bodies\text}, Math. Scand., {\bf 10} (1962) 17-24.

\bibitem{G02}
R. J. Gardner, \emph{The Brunn-Minkowski inequality\text}, Bull. Amer. Math. Soc., {\bf 39} (2002) 355-405.

\bibitem{Gen}
I. Gentil, \emph{The General Optimal $L^p$-Euclidean logarithmic Sobolev inequality by Hamilton-Jacobi equations\text}, J. Funct. Anal., {\bf 202} (2) (2003) 591-599.

\bibitem{Gross}
L. Gross, \emph{Logarithmic Sobolev inequality\text}, Amer. J. Math., {\bf 97} (1975) 1061-1083.

\bibitem{HS09}
C. Haberl, and F. E. Schuster, \emph{General $L_p$ affine isoperimetric inequalities\text}, J. Diff. Geom., {\bf 83} (2009) 1-26.

\bibitem{HS09b}
C. Haberl, and F. E. Schuster, \emph{Asymmetric affine $L_p$ Sobolev inequalities\text}, J. Funct. Anal., {\bf 257} (2009) 641-658.

\bibitem{HSX12}
C. Haberl, F. E. Schuster, and J. Xiao, \emph{An asymmetric affine P\'olya-Szeg\"o principle\text}, Math. Ann., {\bf 352} (2012) 517-542.

\bibitem{HJM15}
J. Haddad, C. H. Jim\'enez, and M. Montenegro, \emph{Sharp affine Sobolev type inequalities via the $L_p$ Busemann–Petty centroid inequality\text}, J. Funct. Anal., {\bf 271} (2016) 454-473.

\bibitem{HLYZ05}
D. Hug, E. Lutwak, D. Yang, and G. Zhang, \emph{On the $L_p$ Minkowski problem for polytopes\text}, Discrete Comput. Geom., {\bf 33} (2005) 699-715.

\bibitem{Kaw85}
B. Kawohl, \emph{Rearrangements and convexity of level sets in PDE\text}, In: Lecture Notes in Mathematics, vol. 1150, Springer, Berlin (1985).

\bibitem{Kaw86}
B. Kawohl, \emph{On the isoperimetric nature of a rearrangement inequality and its consequences for some variational problems\text}, Arch. Ration. Mech. Anal., {\bf 94} (1986) 227-243.

\bibitem{Ke06}
S. Kesavan, \emph{Symmetrization and Applications\text}, Series in Analysis 3, World Scientific, Hackensack (2006).

\bibitem{Ludwig}
M. Ludwig, \emph{Minkowski valuations\text}, Trans. Amer. Math. Soc., {\bf 357} (2005) 4191-4213.

\bibitem{Lut86}
E. Lutwak, \emph{On some affine isoperimetrics inequalities\text}, J. Diff. Geom., {\bf 23} (1986) 1-13.

\bibitem{Lut90}
E. Lutwak, \emph{Centroid bodies and dual mixed volumes\text}, Proc. London. Math. Soc., {\bf 60} (1990) 365-391.

\bibitem{Lut93}
E. Lutwak, \emph{The Brunn-Minkowski-Firey Theory I: Mixed volumes and the Minkowski problem\text}, J. Diff. Geom., {\bf 38} (1993) 131-150.

\bibitem{Lut96}
E. Lutwak, \emph{The Brunn-Minkowski-Firey Theory II: Affine and geominimal surface areas\text}, Adv. Math., {\bf 114} (1996) 244-294.

\bibitem{LYZ00}
E. Lutwak, D. Yang, and G. Zhang, \emph{$L^p$ affine isoperimetric inequalities\text}, J. Diff. Geom., {\bf 56} (2000) 111-132.

\bibitem{LYZ02}
E. Lutwak, D. Yang, and G. Zhang, \emph{Sharp affine $L^p$ Sobolev inequalities\text}, J. Diff. Geom., {\bf 62} (2002) 17-38.

\bibitem{LYZ04}
E. Lutwak, D. Yang, and G. Zhang, \emph{On the $L_p$ Minkowski problem\text}, Trans. Amer. Math. Soc., {\bf 356} (2004) 4359-4370.

\bibitem{LYZ06}
E. Lutwak, D. Yang, and G. Zhang, \emph{Optimal Sobolev norms and the $L^p$ Minkowski problem\text}, Int. Math. Res. Not., (2006) Art. ID 62987, 1-21.

\bibitem{Ma}
D. Ma, \emph{Asymmetric anisotropic fractional Sobolev norms\text}, Arch. Math., {\bf 103} (2014) 167-175. 


\bibitem{VHN15a}
V. H. Nguyen, \emph{Improved $L_p-$mixed volume inequality for convex bodies\text}, J. Math. Anal. Appl., {\bf 431} (2015) 1045-1053.

\bibitem{Ober}
M. Ober, \emph{Asymmetric $L_p$ covexification and the convex Lorentz Sobolev inequality\text}, Monatsh. Math., {\bf 179} (2016) 113-127.

\bibitem{Para1}
L. Parapatits, \emph{$SL(n)-$covariant $L_p-$Minkowski valuations\text}, J. Lond. Math. Soc., {\bf 89} (2014) 397-414.

\bibitem{Para2}
L. Parapatits, \emph{$SL(n)-$contravariant $L_p-$Minkowski valuations\text}, Trans. Amer. Math. Soc., {\bf 366} (2014) 1195-1211.

\bibitem{Pet61}
C. M. Petty, \emph{Centroid surfaces\text}, Pacific J. Math., {\bf 11} (1961) 1535-1547.

\bibitem{Pet71}
C. M. Petty, \emph{Isoperimetric problems\text}, Proc. Conf. Convexity and Combinatorial Geometry (Univ. Oklahoma, 1971), University of Oklahoma, 1972, 26-41.

\bibitem{PS51}
G. P\'olya, and G. Szeg\"o, \emph{Isoperimetric inequalities in mathematical physics\text}, Ann. Math. Stud., {\bf 27} Princeton University Press (1951).

\bibitem{RS}
R. Schneider, \emph{Convex body: The Brunn-Minkowski Theory\text}, in: Encyclopedia of Mathematics and its Applications, vol. 44, Cambridge University press, Cambridge, 1993.

\bibitem{Schuster}
F. E. Schuster, and M. Weberndorfer, \emph{Volume inequalities for asymmetric Wulff shapes\text}, J. Diff. Geom., {\bf 92} (2012) 263-283.

\bibitem{Ta76}
G. Talenti, \emph{Best constants in Sobolev inequality\text}, Ann. Math. Pure Appl., {\bf 110} (1976) 353-372.

\bibitem{Ta93}
G. Talenti, \emph{On isoperimetric theorems in mathematical physics\text} In: P. M. Gruber, and J. M. Wills (eds), Handbook of convex geometry, North-Holland, Amsterdam (1993).

\bibitem{TW12}
T. Wang, \emph{The affine Sobolev-Zhang inequality on $BV(\R^n)$\text}, Adv. Math., {\bf 230} (2012) 2457-2473.

\bibitem{TW13}
T. Wang, \emph{The affine P\'olya-Szeg\"o principle: Equality cases and stability\text}, J. Funct. Anal., {\bf 265} (2013) 1728-1748.

\bibitem{TW15}
T. Wang, \emph{On the discrete functional $L_p$ Minkowski problem\text}, Int. Math. Res. Not., (2015) 10563-10585.

\bibitem{Web}
M. Weberndorfer, \emph{Shadow systems of asymmetric $L_p$ zonotopes\text}, Adv. Math., {\bf 240} (2012) 613-635.

\bibitem{Zhai11}
Z. Zhai, \emph{Notes on affine Gagliardo-Nirenberg inequalities\text}, Potential Anal., (2011) 1-12.

\bibitem{GZ99}
G. Zhang, \emph{The affine Sobolev inequality\text}, J. Diff. Geom., {\bf 53} (1999) 183-202.

\end{thebibliography}
\end{document}